\documentclass{article}
\usepackage{graphicx}
\usepackage{tikz}
\usepackage{color}
\usepackage{amsmath}
\usepackage{amssymb}
\usepackage{amsthm}
\theoremstyle{plain}
\newtheorem{lemma}{Lemma}
\theoremstyle{plain}
\newtheorem{theorem}{Theorem}
\theoremstyle{plain}
\newtheorem{corollary}{Corollary}
\theoremstyle{definition}
\newtheorem{example}{Example}
\theoremstyle{definition}
\newtheorem{remark}{Remark}
\usepackage[margin={1in}]{geometry}
\begin{document}

\title{CHAOTIC BEHAVIOR OF UNIFORMLY CONVERGENT NONAUTONOMOUS SYSTEMS WITH RANDOMLY PERTURBED TRAJECTORIES}

\author{LESZEK SZA\L{}A\footnote{leszek.szala@math.slu.cz}}


\maketitle

\begin{abstract}
We study nonautonomous discrete dynamical systems with randomly perturbed trajectories. We suppose that such a system is generated by a sequence of continuous maps which converges uniformly to a map $f$. We give conditions, under which a recurrent point of a (standard) autonomous discrete dynamical system generated by the limit function $f$ is also recurrent for the nonautonomous system with randomly perturbed trajectories. We also provide a necessary condition for a nonautonomous discrete dynamical system to be nonchaotic in the sense of Li and Yorke with respect to small random perturbations.
\end{abstract}

\section{Introduction}
We consider discrete dynamical systems generated by continuous functions defined on the
Cartesian product $I^{m}$ of $m$ intervals $I=[0,1]$, where $m$ is a positive integer.
Their values are subjected to small random perturbations.
Nonautonomous discrete dynamical systems (with no perturbations) have been recently studied because of their applications, e.g., in biology (\cite{Sen:The}, \cite{Elyadi:Global}, \cite{Wright:Periodic}), medicine (\cite{Coutinho:Threshold} and \cite{Lou:Threshold}), economy (\cite{Zhang:Discrete}), physics (\cite{Joshi:Nonlinear}).
In our systems perturbations are involved, because in practical situations they
often exist.
By $\mathcal{C}(X)$ we denote the set of all continuous functions $f\colon\ X\to X$, where $X$ is a compact metric space.
A sequence of functions $(f_{n})_{n=0}^{\infty}$ is denoted by $f_{0,\infty}$.
If such a sequence converges uniformly to the limit function $f$, we denote it by $f_{n}\rightrightarrows f$.
The main aim of this paper is to study, whether a recurrent point of a standard  (i.e., autonomous with no perturbations) discrete dynamical system $(I^{m},f)$ with $f\in\mathcal{C}(I^{m})$ remains recurrent for a system generated by a sequence $f_{0,\infty}$ in $\mathcal{C}(I^{m})$, where $f_{n}\rightrightarrows f$ and a random perturbation is added to every iteration.
The assumption about uniform convergence is common when nonautonomous discrete dynamical systems are studied.
For example \cite{Kolyada:Topological} showed that the topological entropy of a system $(X,f_{0,\infty})$, where $X$ is a compact metric space and $f_{0,\infty}$ is a sequence of continuous selfmaps of $X$ converging uniformly to $f$, is less than or equal to the topological entropy of $(X,f)$.
Notice that this inequality does not hold if the convergence is not uniform (see \cite{Balibrea:Weak}).
If, additionally, $X=I$, the elements of $f_{0,\infty}$ are surjective and the topological entropy of $(I,f)$ equals zero, then every infinite $\omega$-limit set of $(I,f)$ is an $\omega$-limit set of $(I,f_{0,\infty})$ (see \cite{Stefankova:Inheriting}).

Throughout this paper $\mathbb{N}$ is the set of positive integers and $\mathbb{N}_{0}=\mathbb{N}\cup\{0\}$.
The $i$-th coordinate of $x\in\mathbb{R}^{m}$ is denoted by $x^{(i)}$.
Let $n\in\mathbb{N}$.
The {\itshape $n$-th iteration} of $f$ is defined by $f^{n}(x)=f(f^{n-1}(x))$ and $f^{0}(x)=x$.
Let $\delta>0$.
We call $(x_{i})_{i=0}^{n}$ a {\itshape $\delta$-chain} if $|f(x_{i})-x_{i+1}|<\delta$ for each $i=0,\ldots,n-1$.
If $f_{0,\infty}=(f_{0},f_{1},f_{2},\ldots)$ is a sequence in $\mathcal{C}(X)$, a {\itshape nonautonomous discrete dynamical system} is a pair $(X,f_{0,\infty})$.
The {\itshape trajectory} of $x_{0}$ under $f_{0,\infty}$ is the sequence $(x_{n})_{n=0}^{\infty}$ defined by $x_{n+1}=f_{n}(x_{n})$ for each $n\in\mathbb{N}_{0}$. A discrete (autonomous) dynamical system $(X,f)$, with $f\in\mathcal{C}(X)$ is a particular case of a nonautonomous system $(X,f_{0,\infty})$ with $f_{0,\infty}=(f,f,\ldots)$.
When we deal with recurrence, sometimes it will be necessary to remove a certain number of first elements of $f_{0,\infty}$. For such a sequence with first $k$ elements removed we use a symbol $f_{k,\infty}$.
Whenever we mention the notion of random variables, we assume that they are defined on $\Omega$ where $(\Omega,\Sigma,P)$ is a fixed probability space.
By $\|\cdot\|$ we mean the maximum norm defined on $\mathbb{R}^{m}$, i.e., $\|x\|=\max_{i=1,\ldots,m}|x^{(i)}|$ for $x\in\mathbb{R}^{m}$. The open ball of a radius $r>0$ centered at $x$ is denoted by $B(x,r)$.

A family $\mathcal{F}\subseteq\mathcal{C}(I^{m})$ is {\itshape equicontinuous} at $x\in I^{m}$ if for each $\varepsilon>0$ there is $\delta>0$ such that for each $f\in\mathcal{F}$ and each $y\in I^{m}$, $\|x-y\|<\delta$ implies $\|f(x)-f(y)\|<\varepsilon$. For convenience we recall the well-known Ascoli theorem (see, e.g., \cite{Dieudonne:Foundations}), which is used in the proofs presented in this paper.

\begin{theorem}[Ascoli]
Let $E$ be a compact metric space, $F$ be a Banach space and $\mathcal{C}_{F}(E)$ be the space of all continuous functions defined on $E$ with values in $F$.
Then the closure of $H\subseteq\mathcal{C}_{F}(E)$ is compact if and only if $H$ is equicontinuous and the closure of $\{f(x),f\in H\}$ is compact for each $x\in E$.
\end{theorem}

The results presented in this paper mainly concern $(f_{0,\infty},\delta)$-recurrence for the case of (nonautonomous) $(f_{0,\infty},\delta)$-processes. It is a generalization of $(f,\delta)$-recurrence, introduced in \cite{Szala:Recurrence} for $(f,\delta)$-processes, which have been studied in literature, e.g., by \cite{Jankova:ChaosIn}, \cite{Jankova:Maps} and \cite{Jankova:Systems}. Let $f_{0,\infty}$ be a sequence in $\mathcal{C}(I^{m})$, $m\in\mathbb{N}$ and $\delta>0$. In order to define an $(f_{0,\infty},\delta)$-process it is necessary to consider continuous extensions of all $f_{0}$, $f_{1}$, ...
Let $g$ be any of these functions.
Then we extend its domain to $\mathbb{R}^{m}$ in such a way that $g(\mathbb{R}^{m}\setminus I^{m})\subseteq g(\partial I^{m})$, where $\partial A$ denotes the boundary of $A$.
In order to keep the notation simple, we denote this extension by $g$, as well.
An {\itshape $(f_{0,\infty},\delta)$-process that begins at $x_{0}$} is a sequence of random variables defined by the formula $X_{n+1} = f_{n}(X_{n})+(\xi^{(1)}_{n},\ldots,\xi^{(m)}_{n})$, $n\in\mathbb{N}_{0}$ and $X_{0} = x_{0}$, where all $\xi^{(j)}_{k}$, $j=1,\ldots,m$, $k=0,1,\ldots$ are independent and have uniform continuous distributions on $[-\delta,\delta]$.
Then a point $x\in I^{m}$ is called {\itshape $(f_{0,\infty},\delta)$-recurrent} if for each open neighborhood $U$ of $x$ and each $\delta'\in(0,\delta)$,
\begin{eqnarray}
\label{recurrent-def}
P\left(\bigcap_{n=1}^{\infty}\bigcup_{k=n}^{\infty}\{X_{k}\in U\}\right)=1,
\end{eqnarray}
where $(X_{n})$ is any $(f_{0,\infty},\delta')$-process that begins at $x$.
There is a link between the standard notion of recurrence for standard discrete dynamical systems (with no perturbations) and $(f_{0,\infty},\delta)$-recurrence. In the first case, in each neighborhood of the point, which is said to be recurrent, there are infinitely many points of its trajectory. In the second case, infinitely many of the events $\{X_{k}\in U\}$, $k=1,2,\ldots$ occur with probability one.

Recall the definition of chaos in the sense of Li and Yorke. A function $f\in\mathcal{C}(I)$ is {\itshape chaotic in the sense of Li and Yorke} if there is an uncountable set $S\subseteq I$ such that for all $x,y\in S$ and $x\neq y$,
\begin{eqnarray*}
\liminf\limits_{n\to\infty}|f^{n}(x)-f^{n}(y)|=0,\ \limsup\limits_{n\to\infty}|f^{n}(x)-f^{n}(y)|>0.
\end{eqnarray*}
\newline
Recall that a point $x\in I^{m}$, where $m\in\mathbb{N}$, is a {\itshape fixed point} of $f\in\mathcal{C}(I^{m})$ if $f(x)=x$. We denote the set of all fixed points of $f$ by $\mathrm{Fix}(f)$. A point $x\in I^{m}$ is {\itshape periodic} with {\itshape period $n$} if $f^{n}(x)=x$ and $f^{k}(x)\neq x$ for any $k=1,\ldots,n-1$. We denote the set of all periodic points of $f$ by $\mathrm{Per}(f)$.
\newline
We say that $f$ is {\itshape nonchaotic} if for each $x\in I$ and each $\varepsilon>0$ there exists periodic point $p$ of $f$ such that $\limsup_{n\to\infty}|f^{n}(x)-f^{n}(p)|<\varepsilon$.
It is well-known, that $f$ is either Li-Yorke chaotic, or it is nonchaotic in the sense of the previous definition, which gives a dichotomy between chaos in the sense of Li and Yorke and the simplicity of any orbit of $f$ (see \cite{Smital:Chaotic} and \cite{Jankova:Characterization}).
\newline
The result providing the dichotomy between simplicity and chaoticity mentioned above was generalized for the case of some nonautonomous systems by \cite{Canovas:Li} as described in the following.
Let $f_{0,\infty}$ be a sequence in $\mathcal{C}(I)$ converging uniformly to $f$. Define $F_{n}(x)=f_{n}\circ\ldots\circ f_{0}(x)$ for each $x\in I$. Then $f_{0,\infty}$ is called {\itshape chaotic in the sense of Li and Yorke} if there is an uncountable set $S\subseteq I$ such that for all $x,y\in S$ and $x\neq y$, $\liminf_{n\to\infty}|F^{n}(x)-F^{n}(y)|=0$ and $\limsup_{n\to\infty}|F^{n}(x)-F^{n}(y)|>0$. \cite{Canovas:Li} uses so called pseudoperiodic points of $f_{0,\infty}$ when approximating orbits of $f_{0,\infty}$. If $f_{n}\rightrightarrows f$, these pseudoperiodic points are just periodic points of $f$. We use the same method when approximating orbits of nonautonomous systems with random perturbations. We obtain a generalization of the result published by \cite{Jankova:ChaosIn}, who studied connections between chaotic properties of $f\in\mathcal{C}(I)$ and chaotic behavior of $(f,\delta)$-processes.
We call $f$ {\itshape nonchaotic stable} if for each $\varepsilon>0$ there exists $\delta>0$ such that for each $g\in\mathcal{C}(I)$ and each $x\in I$, $\|f-g\|<\delta$ implies $\limsup_{n\to\infty}|g^{n}(x)-g^{n}(p)|<\varepsilon$ for some periodic point $p$ of $g$. We call $f$ {\itshape nonchaotic with respect to small random perturbations} if for each $\varepsilon>0$ there exists $\delta>0$ such that for each $\delta'\in (0,\delta)$ and each $x_{0}\in I$,
\begin{eqnarray}
P\left(\exists p\in\mathrm{Per}(f)\colon\ \limsup_{n\to\infty}\left|X_{n}-f^{n}(p)\right|<\varepsilon\right)=1,
\label{nonchaotic_with_respect_to_random_perturbations}
\end{eqnarray}
where $(X_{n})$ is arbitrary $(f,\delta')$-process which begins at $x_{0}$.
\cite{Jankova:ChaosIn} proved that $f$ is nonchaotic with respect to small random perturbations provided that $f$ is nonchaotic stable and also mentioned that the opposite implication is not true. We show an analogous theorem, which concerns nonautonomous systems with randomly perturbed trajectories.

\section{Recurrent points}
Recall that a fixed point $x$ is {\itshape attractive} if there is a neighborhood $U$ of $x$ such that for each $y\in U$, $\lim_{n\to\infty}f^{n}(y)=x$.

Briefly speaking, Theorem \ref{fixed_points} states, that each attractive fixed point of the limit function $f$ is also recurrent in the sens of uniformly convergent nonautonomous system connected to $f$ under a small additive stochastic perturbation. Notice that even for autonomous discrete dynamical systems such a statement is not true if the fixed point is not attractive.

\begin{theorem}
Let $f_{0,\infty}$ be a sequence in $\mathcal{C}(I^{m})$, $m\in\mathbb{N}$.
Assume that $f_{0,\infty}$ converges uniformly to a function $f\in\mathcal{C}(I^{m})$.
Let $x_{0}\in I^{m}$ be an attractive fixed point of $f$.
Then there exist $K\in\mathbb{N}$ and $\delta>0$ such that for each integer $k>K$, $x_{0}$ is $(f_{k,\infty},\delta)$-recurrent.
\label{fixed_points}
\end{theorem}
\begin{proof}
Since $x_{0}$ is an attractive fixed point of $f$, there is $\kappa>0$ such that for each $x\in B(x_{0},\kappa)$, $\lim_{n\to\infty}f^{n}(x)=x_{0}$.
By Lemma 1 in \cite{Szala:FC} there is $N\in\mathbb{N}$ such that
\begin{eqnarray*}
\forall x\in B(x_{0},\kappa)\ \exists j\in\{0,\ldots, N\}\colon\ f^{j}(x)\in B(x_{0},\kappa/3).
\end{eqnarray*}
First we show, that there is $K\in\mathbb{N}$ such that for each integer $k>K$, for each $x\in B(x_{0},\kappa)$ and each sequence $(y_{n})$ defined by $y_{0}=x$ and $y_{i+1}=f_{k+i}(y_{i})$ with $i=0,1,\ldots$,
\begin{eqnarray}
\exists j\in\{0,\ldots,N\}\colon\ y_{j}\in B(x_{0},2\kappa/3).\label{twierdzenie1_istnienie}
\end{eqnarray}
It is sufficient to show that there exists $K\in\mathbb{N}$ such that for each integer $k>K$ and for each $j\in\{0,\ldots,N\}$,
\begin{eqnarray}
\left\|f_{k+j-1}\circ\ldots\circ f_{k+1}\circ f_{k}(x)-f^{j}(x)\right\|<\kappa/3.\label{nier6}
\end{eqnarray}
Then inequality (\ref{twierdzenie1_istnienie}) is a direct consequence of (\ref{nier6}) and the triangular inequality.
Let $j$ be any of $\{0,\ldots,N\}$.
Inequality (\ref{nier6}) holds if
\begin{eqnarray}
\left\|f_{k+j-1}\left(f_{k+j-2}\circ\ldots\circ f_{k}(x)\right)-f_{k+j-1}\left(f_{k+j-1}^{j-1}(x)\right)\right\|<\kappa/6\label{nier2}
\end{eqnarray}
and
\begin{eqnarray}
\left\|f_{k+j-1}^{j}(x)-f^{j}(x)\right\|<\kappa/6.\label{nier3}
\end{eqnarray}
Since $f_{n}\rightrightarrows f$,
there is $K_{1}\in\mathbb{N}$ such that for each integer $k>K_{1}$ inequality (\ref{nier3}) holds. 
Since $f_{k+j-1}$ is continuous,
there is $\varepsilon_{1}\in (0,\kappa/6)$ (by Ascoli Theorem it does not depend on $k$ and $j$) such that inequality (\ref{nier2}) holds if
\begin{eqnarray*}
\|f_{k+j-2}\circ\ldots\circ f_{k}(x)-f_{k+j-1}^{j-1}(x)\|<\varepsilon_{1}.
\end{eqnarray*}
This inequality holds if
\begin{eqnarray}
\left\|f_{k+j-2}\circ\ldots\circ f_{k+1}\circ f_{k}(x)-f_{k+j-2}\left(f_{k+j-1}^{j-2}(x)\right)\right\|<\varepsilon_{1}/2\label{nier4}
\end{eqnarray}
and
\begin{eqnarray}
\left\|f_{k+j-2}\left(f_{k+j-1}^{j-2}(x)\right)-\left(f_{k+j-1}\left(f_{k+j-1}^{j-2}(x)\right)\right)\right\|<\varepsilon_{1}/2.\label{nier5}
\end{eqnarray}
Since $f_{0,\infty}$ is a Cauchy sequence, there is
integer $K_{2}>K_{1}$ such that for each integer $k>K_{2}$ inequality (\ref{nier5}) holds.
For inequality (\ref{nier4}) we use continuity of $f_{k+j-2}$ in the same way as above. It implies the existence of $\varepsilon_{2}\in (0,\varepsilon_{1}/2)$ such that (\ref{nier4}) holds if
\begin{eqnarray*}
\left\|f_{k+j-3}\circ\ldots\circ f_{k}(x)-f_{k+j-1}^{j-2}(x)\right\|<\varepsilon_{2}.
\end{eqnarray*}
Using the same procedure several times we obtain the
sequence of positive integers $K_{1}<K_{2}<\ldots<K_{j}$.
Choose $K=K_{j}$.
In order to show $(f_{k,\infty},\delta)$-recurrence we use Borel-Cantelli Lemma in the same way as in the proof of Theorem 2 in \cite{Szala:Recurrence}.
We can replace $f$ by elements of $f_{0,\infty}$ in the proof mentioned above, since $\{f,f_{0},f_{1},\ldots\}$ is an equicontinuous family, which follows by Ascoli theorem.
\end{proof}

The result from Theorem \ref{fixed_points} can be
generalized to the case of attractive periodic points.
The proof is evident and is omitted.
Recall that a periodic point with period $n$ is {\itshape attractive} if it is an attractive fixed point of $f^{n}$.

\begin{theorem}
\label{Tw_punkty_okresowe}
Let $f_{0,\infty}$ be a sequence in $\mathcal{C}(I^{m})$, $m\in\mathbb{N}$.
Assume that $f_{0,\infty}$ converges uniformly to a function $f\in\mathcal{C}(I^{m})$.
Let $x_{0}\in I^{m}$ be an attractive periodic point of $f$ with period $n$.
Then there exist $K\in\mathbb{N}$ and $\delta>0$ such that for each integer $k>K$, $x_{0}$ is $(f_{k,\infty},\delta)$-recurrent.
\end{theorem}

The following example shows that the assumption about uniform convergence in Theorem \ref{fixed_points} is necessary.

\begin{example}
\label{przyklad}
\noindent
Let $f\in\mathcal{C}(I)$ be defined by the formula
\begin{eqnarray*}
f(x)=\begin{cases}0, & x\in\left[0,\frac{1}{2}\right], \\ 4x-2, & x\in\left(\frac{1}{2},\frac{3}{4}\right], \\ 1, & x\in\left(\frac{3}{4},1\right]
\end{cases}
\end{eqnarray*}
(see Figure \ref{wykres4}).

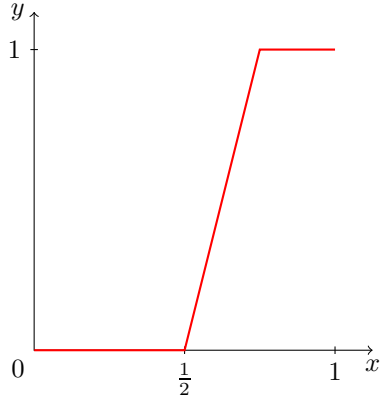
\begin{figure}[htb]
\begin{center}
\begin{tikzpicture}
\draw [<->] (0,4.5) node [left] {$y$} -- (0,0) -- (4.5,0) node [below] {$x$};
\draw (-0.05,4) node [left] {1} -- (0.05,4);
\draw (0,0) node [below left] {0} -- (0,0);
\draw (2,-0.05) node [below] {$\frac{1}{2}$}-- (2,0.05);
\draw (4,-0.05) node [below] {1} -- (4,0.05);
\draw [red, thick] (0,0) -- (1,0) -- (2,0) -- (3,4) -- (4,4);
\end{tikzpicture}
\caption{The graph of $f$ from Example \ref{przyklad}.}
\label{wykres4}
\end{center}
\end{figure}

\noindent
Obviously $x_{0}=0$ is an attractive fixed point of $f$.
Consider a sequence $f_{0,\infty}$ in $\mathcal{C}(I)$ defined for each $n\in\mathbb{N}_{0}$ by
\begin{eqnarray*}
f_{n}\left(x\right)=\begin{cases}8\cdot 2^{n}x, & x\in \left[0,\frac{1}{8\cdot 2^{n}}\right], \\ -8\cdot 2^{n}x+2, & x\in \left(\frac{1}{8\cdot 2^{n}},\frac{1}{4\cdot 2^{n}}\right],\\ 0, & x\in\left(\frac{1}{4\cdot 2^{n}},\frac{1}{2}\right], \\ 4x-2, & x\in \left(\frac{1}{2},\frac{3}{4}\right],\\ 1, & x\in\left(\frac{3}{4},1\right]\end{cases}
\end{eqnarray*}
(the graphs of $f_{0}$ and $f_{1}$ are sketched on Figure $\ref{wykres2}$).
Then
$f_{0,\infty}$ converges pointwise to $f$,
but the convergence is not uniform.

\noindent
In order to show that $0$ is not $(f_{k,\infty},\delta)$-recurrent 
for any $\delta>0$ and any $k\in\mathbb{N}$,
let $\delta\in (0,1/5)$ and $k\in\mathbb{N}_{0}$. Define $X_{0}=x_{0}$ and $X_{n+1}=f_{k+n}(X_{n})+\xi_{n}$, where $n\in\mathbb{N}_{0}$ and $(\xi_{0},\xi_{1},\ldots)$ is a sequence of random variables, which are independent and have uniform continuous distributions on $(-\delta,\delta)$. Then $P(X_{2}\in [4/5,1])>0$, which implies
\begin{eqnarray*}
P\left(\bigcap_{k=2}^{\infty}\left\{X_{k}\in\left[\frac{4}{5},1\right]\right\}\right)>0.
\end{eqnarray*}

\begin{figure}[htb]
\begin{center}
\begin{tikzpicture}
\draw [<->] (0,4.5) node [left] {$y$} -- (0,0) -- (4.5,0) node [below] {$x$};
\draw (-0.05,4) node [left] {1} -- (0.05,4);
\draw (0,0) node [below left] {0} -- (0,0);
\draw (2,-0.05) node [below] {$\frac{1}{2}$}-- (2,0.05);
\draw (4,-0.05) node [below] {1} -- (4,0.05);
\draw [red, thick] (0,0) -- (0.5,3.9) -- (1,0) -- (2,0) -- (3,4) -- (4,4);
\draw [<->] (6,4.5) node [left] {$y$} -- (6,0) -- (10.5,0) node [below] {$x$};
\draw (5.95,4) node [left] {1} -- (6.05,4);
\draw (6,0) node [below left] {0} -- (6,0);
\draw (8,-0.05) node [below] {$\frac{1}{2}$}-- (8,0.05);
\draw (10,-0.05) node [below] {1} -- (10,0.05);
\draw [red, thick] (6,0) -- (6.25,3.9) -- (6.5,0) -- (8,0) -- (9,4) -- (10,4);

\end{tikzpicture}
\caption{The graphs of $f_{0}$ and $f_{1}$ from Example \ref{przyklad}.}
\label{wykres2}
\end{center}
\end{figure}
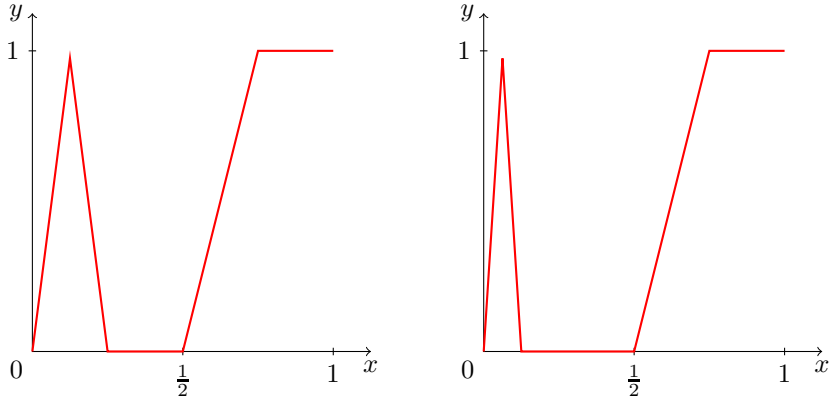
\end{example}

Recall that the {\itshape $\omega$-limit set} of a point $x\in I$ for $f\in\mathcal{C}(I)$ is the set of all limit points of the sequence $(f^{n}(x))_{n=0}^{\infty}$. We say that an $\omega$-limit set $\tilde{\omega}$ is {\itshape maximal} if for each $\omega$-limit set $\tilde{\omega}_{1}$, $\tilde{\omega}_{1}\subseteq\tilde{\omega}$ or $\tilde{\omega}_{1}\cap\tilde{\omega}=\emptyset$.
We say that $f$ is of {\itshape type $2^{\infty}$} if it has a periodic point of period $2^{n}$ for all $n\in\mathbb{N}$ and no periodic points of other periods.
Let $f\in\mathcal{C}(I)$ and $A\subseteq I$. We say that $A$ is {\itshape invariant} for $f$ if $f(A)\subseteq A$.
The following theorem
is
formulated in dimension one only, because in this case an infinite $\omega$-limit set of a function of type $2^{\infty}$ has a special structure described in the following lemma (see Lemma 3.1 and Theorem A \cite{Fedorenko:Characterization}).

\begin{lemma}
\label{rozklad_na_czesci_okresowe}
Let $f\in\mathcal{C}(I)$ be of type $2^{\infty}$.
Let $f$ have an infinite $\omega$-limit set $\tilde{\omega}$.
Then for each $k\in\mathbb{N}$ there is a decomposition $\{M(i,k),i=1,\ldots,2^{k}\}$ of $\tilde{\omega}$ such that every two sets $M(i,k)$, $M(j,k)$, where $i\neq j$, are separated by disjoint compact intervals and $f(M(i,k))=M(i+1,k)$ for any $i(\mathrm{mod}2^{k})$ .
\end{lemma}

The following result provides another class of points, that are known to be recurrent for both, standard discrete dynamical systems (without perturbations) and $(f_{0,\infty},\delta)$-processes, provided that $\delta>0$ is small enough.

\begin{theorem}
\label{tw_gl_1}
Let $f\in\mathcal{C}(I)$ be of type $2^{\infty}$ and have a maximal infinite $\omega$-limit set $\tilde{\omega}$.
Assume that the minimal closed and invariant interval $V$ containing $\tilde{\omega}$ contains exactly one periodic orbit of period $2^{n}$ for each $n\in\mathbb{N}$ and this periodic orbit is not attractive.
Let $f_{0,\infty}$ be a sequence in $\mathcal{C}(I)$, which converges uniformly to $f$.
Then there exist $K\in\mathbb{N}$ and $\delta>0$ such that for each integer $k>K$, every point $x\in\tilde{\omega}$ is $(f_{k,\infty},\delta)$-recurrent.
\end{theorem}
\begin{proof}
Let $U=[u,v]$ be the convex hull of $\tilde{\omega}$.
By Lemma 3.4. in \cite{Fedorenko:Characterization} there is interval $J\supseteq U$ relatively open in $[0,1]$ such that $f(\overline{J})\subseteq J$.
First we show the following condition:
\begin{eqnarray}
\forall x\in J\ \exists n\in\mathbb{N}\colon\ f^{n}(x)\in U.
\label{tw_gl_war_1}
\end{eqnarray}
To see this assume that the opposite of (\ref{tw_gl_war_1}) is true.
Then $(f^{n}(x))$ has a subsequence that converges to a point of $U$.
This limit point can be $u$ or $v$ only.
Without loss of generality suppose this limit point is $u$.
Since $u\in\tilde{\omega}$ and it is not a periodic point, there is $\varepsilon_{u}>0$ and $n\in\mathbb{N}$ such that $f^{n}(B(u,\varepsilon_{u}))\subset (u,v)$. There is $m\in\mathbb{N}$ such that $f^{m}(x)\in B(u,\varepsilon_{u})$. Then $f^{m+n}(x)\in (u,v)$, which is a contradiction.
Therefore statement (\ref{tw_gl_war_1}) is proved.
\newline
Choose $x_{0}\in\tilde{\omega}$.
Assume that $x_{0}$ is not an isolated point of $\tilde{\omega}$.
Without loss of generality assume that $u\neq 0$, $v\neq 1$ and $J$ is open.
Let $\kappa>0$ be such that $B(x_{0},\kappa)\subset J$.
We show that
\begin{eqnarray}
\forall x\in J\ \forall\delta'>0\ \exists n\in\mathbb{N}\ \exists \delta'\textrm{-chain } (z_{0},\ldots,z_{n})\colon\ z_{0}=x, z_{n}\in B(x_{0},\kappa/4).
\label{delta-lancuch}
\end{eqnarray} 
By (\ref{tw_gl_war_1}) it is enough to prove the existence of a $\delta'$-chain mentioned above for each $x\in U$.
Choose $\delta'>0$.
Let $U(i,k)$ be the convex hull of $M(i,k)$ for each $i=1,\ldots,2^{k}$ and each $k=1,2,\ldots$, where $M(i,k)$ are as in Lemma \ref{rozklad_na_czesci_okresowe}.
Each of these convex hulls can be written as follows:
\begin{eqnarray*}
U(i,k)=U(i_{1},k+1)\cup R(i,k)\cup U(i_{2},k+1)
\end{eqnarray*}
for some interval $R(i,k)$.
It is obvious, that there is $M>0$ such that for each $k\ge M$, the diameter of at least one of the sets $U(i,k)$, $i=1,\ldots,2^{k}$ is smaller than $\delta'$.
In order to prove (\ref{delta-lancuch}) we consider the following cases:
\begin{enumerate}
\item[(a)]
Since the endpoints of every $U(i,k)$ are elements of $\tilde{\omega}$, the  existence of $\delta'$-chain mentioned in (\ref{delta-lancuch}) is obvious whenever
\begin{eqnarray*}
x\in\bigcup_{k=M}^{\infty}\bigcup_{i=1}^{2^{k}}U(i,k).
\end{eqnarray*}
\item[(b)]
Assume that $x$ is not a member of any $U(i,k)$ with $k\ge M$ and it is not a member of any $R(i,j)$. Then there is a sequence of positive integers $(i_{k})$ such that $i_{k}\in\{1,\ldots,2^{k}\}$ for each $k=1,2,\ldots$ and
\begin{eqnarray*}
x\in\bigcap_{k=1}^{\infty}U(i_{k},k).
\end{eqnarray*}
This leads to a contradiction.
\item[(c)]
Assume that $x$ is not a member of any $U(i,k)$ with $k\ge M$, that it is a member of some $R(i,j)$ and that it is a periodic point. Without loss of generality we assume $x\in R(1,1)$. Then $x$ is a fixed point. Since $x$ is not attractive, in every neighborhood of $x$ there is a point, whose trajectory does not converge to $x$. Let $y$ be such a point in $B(x,\delta')$. Then $(f^{n}(y))$ has a subsequence which converges to a point of $\tilde{\omega}$ (which finishes this part of the proof) or converges to another periodic point with a different period. By the intermediate value theorem there is $z$ between $x$ and $y$ such that its orbit intersects $\bigcup_{i=1}^{2^{k}}U(i,k)$ for some $k\ge M$. Then (a) applies.
\item[(d)]
Assume that $x$ is not a member of any $U(i,k)$ with $k\ge M$, it is a member of some $R(i,j)$ and it is not a periodic point. Then $(f^{n}(x))$ has a subsequence which converges to a point of $\tilde{\omega}$ (which finishes this part of the proof) or converges to a periodic point. Denote this periodic point by $y$. Then the trajectory of $x$ intersects $B(y,\delta'/2)$. Since $y$ is not attractive, in its every neighborhood (in particular in $B(y,\delta'/2)$) there is $z$ such that $(f^{n}(z))$ has a subsequence which converges to a point of $\tilde{\omega}$ (which finishes this part of the proof) or converges to another periodic point with a different period (then we use (c)).
\end{enumerate}
Then, property (\ref{delta-lancuch}), i.e., the existence of a $\delta'$-chain from each point of $x\in J$ to $B(x_{0},\kappa/4)$ is proved under the assumption that $x_{0}$ is not an isolated point of $\tilde{\omega}$.
This assumption can be removed due to the structure of $\tilde{\omega}$.
Define $r=\min\{|x-y|\colon\ x\in [0,1]\setminus J, y\in f(\overline{J})\}$.
Choose $\delta\in (0,r)$ and $\delta'\in (0,\delta/2)$.
Let $(x_{n}, n=0,\ldots,k_{x})$ be such a $\delta'$-chain from $x$ to $B(x_{0},\kappa/4)$.
Let $c_{1},\ldots,c_{r_{x}}$ be such that $x_{n+1}=f(x_{n})+c_{n}$, $n=0,\ldots,k_{x}-1$. Continuity of $f$ implies that there is $r_{x}>0$ such that for each $y\in B(x,r_{x})$ and for each $\delta'$-chain $(y_{n})$ defined by $y_{0}=y$ and $y_{n+1}=f(y_{n})+c_{n}$ with $n=0,\ldots,r_{x}-1$ we have $y_{r_{x}}\in B(x_{0},\kappa/2)$.
Continuity of $f$ implies also that there is $\varepsilon_{x}\in (0,\delta')$, such that whenever $d_{n}\in (0,\delta')$ and $|c_{n}-d_{n}|<\varepsilon_{x}$, then for each $z\in B(x,r_{x})$ and for each $\delta$-chain $(z_{n})$ defined by $z_{0}=z$ and $z_{n+1}=f(z_{n})+d_{n}$ for each $n=0,\ldots,r_{x}-1$ we have $z_{r_{x}}\in B(x_{0},3\kappa/4)$.
\newline
Since $\mathcal{A}=\{B(x,r_{x}),x\in\overline{J}\}$ is an open cover of $U$ and $U$ is compact,
a finite subcover of $\mathcal{A}$ can be chosen.
Let $\{B(x,r_{x}),x\in\mathcal{X}\}$ be such a finite subcover and define
\begin{eqnarray*}
N=\max_{x\in\mathcal{X}}k_{x},\ \varepsilon=\min_{x\in\mathcal{X}}\varepsilon_{x}.
\end{eqnarray*}
Notice that $N<\infty$ and $\varepsilon>0$, since $\mathcal{X}$ is a finite set.
\newline
The same manner as in the proof of Theorem \ref{fixed_points} leads to the conclusion that there is $K\in\mathbb{N}$ such that for each integer $k>K$ and each $j\in\{0,\ldots,N\}$,
\begin{eqnarray}
\left\|f_{k+j-1}\circ\ldots\circ f_{k}(x)-f^{j}(x)\right\|<\kappa/4.
\end{eqnarray}
Let $(X_{n})$ be any $(f_{k,\infty},\delta)$-process that begins at any point $x\in J$. For each $l\in\mathbb{N}$,
\begin{eqnarray*}
P\left(X_{lN}\in B(x_{0},\kappa)\right)\ge\left(\frac{\varepsilon}{\delta}\right)^{N}>0.
\end{eqnarray*}
Then $(f_{k,\infty},\delta)$-recurrence follows by Borel-Cantelli lemma. Equicontinuity of $\{f,f_{0},f_{1},\ldots\}$ follows from Ascoli theorem and it is used when replacing $f$ by elements of $f_{0,\infty}$.
\end{proof}

It is well-known that an infinite $\omega$-limit set $\tilde{\omega}$ of a function $f\in\mathcal{C}(I)$ of type $2^{\infty}$ has the following property: there is a sequence of compact periodic intervals $(J_{n})$ such that for each $n\in\mathbb{N}_{0}$, $J_{n}\supset J_{n+1}$, $J_{n}$ is periodic with period $2^{n}$ and
\begin{eqnarray*}
\tilde{\omega}\subseteq\bigcap\limits_{n=1}^{\infty}\bigcup\limits_{k=0}^{2^{n}-1}f^{k}(J_{n})
\end{eqnarray*}
(see \cite{Fedorenko:Characterization}).
If $f\in\mathcal{C}(I)$ is of type $2^{\infty}$ and has two infinite $\omega$-limit sets $\tilde{\omega}_{1}\neq\tilde{\omega}_{2}$, then exactly one of the following three cases holds: (i) $\tilde{\omega}_{1}\subset\tilde{\omega}_{2}$, (ii) $\tilde{\omega}_{2}\subset\tilde{\omega}_{1}$, (iii) $\tilde{\omega}_{1}\cap\tilde{\omega}_{2}=\emptyset$ (see \cite{Schweizer:Measures}).
It is easy to see, that if (iii) holds and $(J^{(1)})$, $(J^{(2)})$ are the sequences of periodic intervals described above for $\tilde{\omega}_{1}$ and $\tilde{\omega}_{2}$ respectively, then there exists $N\in\mathbb{N}$ such that for each integer $n>N$,
\begin{eqnarray*}
\bigcup\limits_{k=0}^{2^{n}-1}f^{k}\left(J^{(1)}_{n}\right)\cap\bigcup\limits_{k=0}^{2^{n}-1}f^{k}\left(J^{(2)}_{n}\right)=\emptyset.
\end{eqnarray*}

The above remarks allow us to formulate Theorem \ref{tw_gl_1} in more general form, namely the assumption that $V$ contains only one maximal infinite $\omega$-limit set can be relaxed.

\begin{corollary}
\label{Wniosek1}
Let $f\in\mathcal{C}(I)$ be of type $2^{\infty}$.
Assume that $f$ has a maximal infinite $\omega$-limit set $\tilde{\omega}$.
Assume that there exists $m\in\mathbb{N}$ such that the minimal closed and invariant interval $V$ containing $\tilde{\omega}$ contains exactly $m$ different maximal $\omega$-limit sets, $m$ periodic orbits of period $2^{n}$ for each $n\in\mathbb{N}$ and none of them is attractive.
Let $f_{0,\infty}$ be a sequence in $\mathcal{C}(I)$, which converges uniformly to $f$.
Then there exist $K\in\mathbb{N}$ and $\delta>0$ such that for each integer $k>K$, every point $x\in\tilde{\omega}$ is $(f_{k,\infty},\delta)$-recurrent.
\end{corollary}

\begin{remark}
Let $f$ and $\tilde{\omega}$ be as above. If $x_{0}$ is an isolated point of $\tilde{\omega}$, then $x_{0}$ is not a recurrent point for $f$.
\end{remark}

\begin{remark}
Let $x_{0}$ and $f$ be as in Theorem \ref{Tw_punkty_okresowe} or Corollary \ref{Wniosek1}. Then for each $k\in\mathbb{N}$ there exists $\delta_{k}>0$ such that $x_{0}$ is $(f^{k},\delta_{k})$-recurrent.
\end{remark}

The following example shows that the statement of Theorem \ref{tw_gl_1} is not true if $f$ has infinitely many infinite $\omega$-limit sets. Example \ref{przyklad_2_2} is a modification of Example 4.1 from \cite{Szala:Recurrence}.

\begin{example}
\label{przyklad_2_2}
Let $\tau(x)=1-|2x-1|$ for each $x\in [0,1]$. Let $g\in\mathcal{C}(I)$ be of type $2^{\infty}$ defined by
\begin{eqnarray*}
g(x)=\begin{cases}\tau^{2}(\lambda ), & x\in[0,\tau(\lambda)], \\ \tau(x), & x\in (\tau(\lambda),1],\end{cases}
\end{eqnarray*}
where $\lambda=0.8249080\ldots$ (see \cite{Misiurewicz:Smooth}, Remark 4).
We use the same notation as in Example 4.1. in \cite{Szala:Recurrence}, i.e.,
$\{M(i,k), i=1,\ldots,2^{k}\}$, $k\in\mathbb{N}$ is a decomposition of $\tilde{\omega}$ into periodic portions of period $2^{k}$ (see Lemma \ref{rozklad_na_czesci_okresowe}),
$U^{(i)}_{k}=[u^{(i)}_{k},v^{(i)}_{k}]$ is a convex hull of $M(i,k)$ where $i=1,\ldots,2^{k}$,
$k\in\mathbb{N}$.
For each integer $k>1$ and $i=1,\ldots,2^{k}$ there is a periodic point $x^{(i)}_{k}$ of period $2^{k}$ in $U_{k}^{(i)}$.
Moreover,
\begin{eqnarray*}
x_{k}^{(i)}\in\left(v_{k+1}^{(2i-1)},u_{k+1}^{(2i)}\right).
\end{eqnarray*}
For each $k\in\mathbb{N}$ and each $i=1,\ldots,2^{k}$ define
\begin{eqnarray*}
\varepsilon_{k}^{(i)}=\min\left\{x_{k}^{(i)}-v_{k+1}^{(2i-1)},u_{k+1}^{(2i)}-x_{k}^{(i)}\right\}.
\end{eqnarray*}
For each $k\in\mathbb{N}$ define
\begin{eqnarray*}
\varepsilon_{k}=\min\left\{\varepsilon_{k}^{(i)},i=1,\ldots,2^{k}\right\}.
\end{eqnarray*}
Let
\begin{eqnarray*}
I_{k}^{(i)}=\left(x_{k}^{(i)}-\varepsilon_{k},x_{k}^{(i)}+\varepsilon_{k}\right).
\end{eqnarray*}
Let $(i_{k})$ be a sequence such that for each $k\in\mathbb{N}$, $i_{k}\in\{1,\ldots,2^{k}\}$ and
\begin{eqnarray*}
\varepsilon_{k}^{(i_{k})}=\varepsilon_{k}.
\end{eqnarray*}
Let $f$ be defined such that
\begin{enumerate}
\item
$f$ equals $g$ on the set $[0,1]\setminus\bigcup\nolimits_{k=1}^{\infty}\bigcup\nolimits_{i=1}^{2^{k}}I_{k}^{(i)}$;
\item
for each ach $k\in\mathbb{N}$ and $i=i_{k}$,
\newline
inside the square
\begin{eqnarray*}
\left[x_{k}^{(i_{k})}-\frac{2}{5}\varepsilon_{k},x_{k}^{(i_{k})}+\frac{2}{5}\varepsilon_{k}\right]\times \left[f\left(x_{k}^{(i_{k})}\right)-\frac{2}{5}\varepsilon_{k},\left(x_{k}^{(i_{k})}\right)+\frac{2}{5}\varepsilon_{k}\right]
\end{eqnarray*}
we define $f$ as a ``diminished copy'' of the graph of $g$,
\newline
on the interval
\begin{eqnarray*}
\left[x_{k}^{(i_{k})}-\frac{4}{5}\varepsilon_{k},x_{k}^{(i_{k})}-\frac{2}{5}\varepsilon_{k}\right)
\end{eqnarray*}
we define $f$ as a constant function equal
\begin{eqnarray*}
f\left(x_{k}^{(i_{k})}-\frac{2}{5}\varepsilon_{k}\right),
\end{eqnarray*}
on the interval
\begin{eqnarray*}
\left(x_{k}^{(i_{k})}+\frac{2}{5}\varepsilon_{k},x_{k}^{(i_{k})}+\frac{4}{5}\varepsilon_{k}\right]
\end{eqnarray*}
we define $f$ as a constant function equal
\begin{eqnarray*}
f\left(x_{k}^{(i_{k})}+\frac{2}{5}\varepsilon_{k}\right);
\end{eqnarray*}
\item
for each $k\in\mathbb{N}$, for each $i\in\{1,\ldots,2^{k}\}\setminus\{i_{k}\}$ and for each
\begin{eqnarray*}
x\in\left[x_{k}^{(i)}-\frac{4}{5}\varepsilon_{k},x_{k}^{(i)}+\frac{4}{5}\varepsilon_{k}\right]
\end{eqnarray*}
we define
\begin{eqnarray*}
f(x)=x+f\left(x_{k}^{(i)}\right)-x_{k}^{(i)};
\end{eqnarray*}
\item
for each $k\in\mathbb{N}$ and for each $i\in\{1,\ldots,2^{k}\}$, $f$ is affine on the intervals
\begin{eqnarray*}
\left[x_{k}^{(i)}-\varepsilon_{k},x_{k}^{(i)}-\frac{4}{5}\varepsilon_{k}\right]
\end{eqnarray*}
and
\begin{eqnarray*}
\left[x_{k}^{(i)}+\frac{4}{5}\varepsilon_{k},x_{k}^{(i)}+\varepsilon_{k}\right].
\end{eqnarray*}
\end{enumerate}

\noindent
Assume that the points of infinite $\omega$-limit set $\tilde{\omega}$ of $f$ are $(f,\delta)$-recurrent for some $\delta>0$.
There exists positive integer $k_{0}>1$ such that $\varepsilon_{k_{0}}<\delta$.
Let $\delta'=2\varepsilon_{k_{0}}/5$.
It follows from the definition that the points of $\tilde{\omega}$ are $(f,\delta')$-recurrent.
Let $(X_{n})$ be an $(f,\delta')$-process that begins at some point $x\in\tilde{\omega}$.
Let $J_{k}=\bigcup_{i=1}^{2^{k}}I_{k}^{(i)}$. Define the events $A_{n}=\{X_{n}\in J_{k_{0}}\}$, $n=1,2,\ldots$.
Notice that if $X_{m}\in J_{k_{0}}$ for some $m\in\mathbb{N}$, then $X_{n}\in J_{k_{0}}$ for each $n\ge m$, i.e.,
$A_{m}\subseteq\bigcap_{n=m}^{\infty}A_{n}$ for each $m\in\mathbb{N}_{0}$. Then
\begin{eqnarray*}
P\left(\bigcup_{m=1}^{\infty}\bigcap_{n=m}^{\infty}A_{n}\right)\ge P\left(\bigcup_{m=1}^{\infty}A_{m}\right)=1,
\end{eqnarray*}
which means that the points of $\tilde{\omega}$ cannot be $(f,\delta)$-recurrent with any $\delta>0$.
Notice that $(f,\delta)$-recurrence is a particular example of $(f_{0,\infty},\delta)$-recurrent with $f_{0,\infty}=(f,f,\ldots)$.
\end{example}

\begin{remark}
Using a similar modification of function $g$ from Example \ref{przyklad_2_2} we can obtain a function $f$ of type $2^{\infty}$ with the following properties: $f$ has exactly one infinite $\omega$-limit set $\tilde{\omega}$; every periodic point of $f$ is a member of an open interval of periodic points with the same period; for each $x\in\tilde{\omega}$ and each $\delta>0$, $x$ is not $(f,\delta)$-recurrent.
\end{remark}

\section{Approximation by periodic orbits of the limit function}
We generalize the definition of nonchaoticity (see Introduction) in the following way: a sequence $f_{0,\infty}$ in $\mathcal{C}(I)$, which converges uniformly to $f\in\mathcal{C}(I)$, is {\itshape nonchaotic with respect to small random perturbations} if for each $\varepsilon>0$ there exist $K\in\mathbb{N}$ and $\delta>0$ such that for each integer $k>K$, for each $\delta'\in (0,\delta)$ and for each $x_{0}\in I$,
\begin{eqnarray}
P\left(\exists p\in\mathrm{Per}(f)\colon\ \limsup_{n\to\infty}\left|X_{n}-f^{n}(p)\right|<\varepsilon\right)=1,
\label{nonchaotic_with_respect_to_random_perturbations_2}
\end{eqnarray}
holds, where $(X_{n})$ is any $(f_{k,\infty},\delta')$-process which begins at $x_{0}$.
Here we provide a sufficient condition for $f_{0,\infty}$ to be nonchaotic with respect to small random perturbations.

\begin{theorem}
Let $f_{0,\infty}$ be a sequence in $\mathcal{C}(I)$ converging uniformly to $f\in\mathcal{C}(I)$. Let $f$ be nonchaotic stable. Then $f_{0,\infty}$ is nonchaotic with respect to small random perturbations.
\end{theorem}
\begin{proof}
Assume that $f$ is nonchaotic stable. It is already known that for each $\varepsilon>0$ there is $\delta>0$ such that for each $x_{0}\in I$ and each $(f,\delta)$-chain $(x_{n})$ starting at $x_{0}$ there is periodic point $p$ of $f$ such that $\limsup_{n\to\infty}|x_{n}-f^{n}(p)|<\varepsilon$ (see \cite{Jankova:ChaosIn}).
Let $\varepsilon>0$ be arbitrary.
Let $\delta$ be as above.
Since $f_{n}\rightrightarrows f$, there is $K\in\mathbb{N}$ such that for each integer $k\ge K$, $\|f_{k}-f\|<\delta/2$. Choose any $k>K$ and $x_{0}\in I$.
Let $(Y_{n})$ by any $(f_{k,\infty},\delta/2)$-process that begins at $x_{0}$.
For each $\omega\in\Omega$, $(Y_{n}(\omega))$ is an $(f,\delta)$-chain (we work with a fixed probability space $(\Omega,\Sigma,P)$ - see Introduction).
Thus there is $p\in\mathrm{Per}(f)$ (which depends on $\omega$) such that $\limsup_{n\to\infty}|Y_{n}(\omega)-p(\omega)|<\varepsilon$.
Then (\ref{nonchaotic_with_respect_to_random_perturbations_2}) holds.
\end{proof}

\begin{remark}
\label{uwaga}
The definition which states whether or not a function is nonchaotic with respect to small random perturbations was introduced in \cite{Jankova:ChaosIn} with the following, slightly different, condition: for each $\varepsilon>0$ there exists $\delta>0$ such that for each $x_{0}\in I$ and each $(f,\delta)$-process $(X_{n})$ starting at $x_{0}$ there is $p\in\mathrm{Per}(f)$ such that
\begin{eqnarray*}
P\left(\limsup_{n\to\infty}\left|X_{n}-f^{n}(p)\right|<\varepsilon\right)=1.
\end{eqnarray*}
In this definition $(X_{n})$ is a sequence of random variables and $p$ has to be a random variable too. Otherwise the following problem may occur: let $f$ be the function presented at Figure \ref{wykres5}. Let $(X_{n})$ be any $(f,\delta)$-process with $x_{0}=0.5$ and $\delta\in (0,0.2)$. Then
\begin{eqnarray*}
P\left(\bigcup_{k=0}^{\infty}\bigcap_{n=k}^{\infty}\left\{X_{n}\in\left[0,\frac{1}{5}\right)\right\}\right)=P\left(\bigcup_{k=0}^{\infty}\bigcap_{n=k}^{\infty}\left\{X_{n}\in \left(\frac{4}{5},1\right]\right\}\right)>0.
\end{eqnarray*}
It is obvious that there is no periodic point of $f$ whose orbit intersects infinitely many times a small neighborhood of 0 and a small neighborhood of 1.

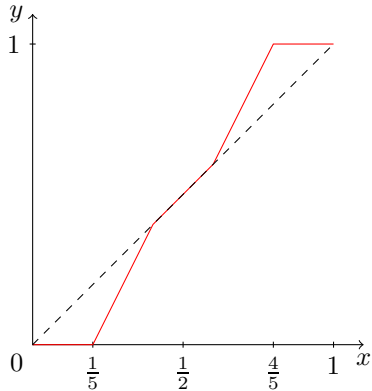
\begin{figure}[htb]
\begin{center}
\begin{tikzpicture}[scale=0.8]
\draw [<->] (0,5.5) node [left] {$y$} -- (0,0) -- (5.5,0) node [below] {$x$};
\draw (-0.05,5) node [left] {1} -- (0.05,5);
\draw (0,0) node [below left] {0} -- (0,0);
\draw (1,-0.05) node [below] {$\frac{1}{5}$}-- (1,0.05);
\draw (2.5,-0.05) node [below] {$\frac{1}{2}$}-- (2.5,0.05);
\draw (4,-0.05) node [below] {$\frac{4}{5}$}-- (4,0.05);
\draw (5,-0.05) node [below] {1} -- (5,0.05);
\draw [red] (0,0) -- (1,0) -- (2,2) -- (3,3) -- (4,5) -- (5,5);
\draw [dashed] (0,0) -- (5,5);
\end{tikzpicture}
\caption{The graph of $f$ in Remark \ref{uwaga}.}
\label{wykres5}
\end{center}
\end{figure}
\end{remark}


\end{document}